\documentclass[12pt]{article}
 \usepackage{a4, latexsym, amsfonts,amsmath}
 
\newtheorem{lemma}{Lemma}[section]

 \newtheorem{con}{Conjecture}[section]
 \newtheorem{open problem}{Open problem}[section]
 \newtheorem{prop}{Proposition}[section]

 \newenvironment{proof}{\trivlist
      \item[\hskip\labelsep
      {\itshape Proof.}]\normalfont}
      {\hspace*{\fill}$\Box$\endtrivlist}

\begin{document}

\title{An elementary proof of a result Ma and Chen\\ }

\author{ Qing Han \thanks{Q. Han is with Faculty of Common Courses, South China Business College of Guangdong university of foreign studies£¬Guangzhou, 510545, China  (email: 46620467@qq.com).  }\,\,  Pingzhi Yuan\thanks{P. Yuan is with School of Mathematics, South China Normal University, Guangzhou 510631, China (email: 20091004@m.scnu.edu.cn).} }

\date{}
\maketitle
 \edef \tmp {\the \catcode`@}
   \catcode`@=11
   \def \@thefnmark {}

    \@footnotetext { Supported by  NSF of China  (Grant No.  11671153).}

\begin{abstract} In 1956, Je$\acute{s}$manowicz conjectured that, for positive integers $m$ and $n$  with $m>n, \, \gcd(m,\, n)=1$ and $m\not\equiv n\pmod{2}$, the exponential Diophantine equation $(m^2-n^2)^x+(2mn)^y=(m^2+n^2)^z$ has only the positive integer solution $(x,\,y,\, z)=(2,\,2,\,2)$. Recently, Ma and Chen \cite{MC17} proved the conjecture if $4\not|mn$ and $y\ge2$. In this paper, we present an elementary  proof of the result of Ma and Chen \cite{MC17}.
\end{abstract}


 \catcode`@=\tmp
 \let\tmp = \undefined

\vskip 3mm
\section{\bf Introduction}
Let $ a,\, b$ and $c$  be  positive integers satisfying $a^2+b^2=c^2$. Such a triple $(a,\,b,\,c)$ is called a $Pythagorean\,\, triple$. If $\gcd(a,\,b,\, c)=1$, this triple is called $primitive$. It is well-known that a primitive Pythagorean triple $(a,\,b,\,c)$ can be parameterized by
$$ a=m^2-n^2, \quad b=2mn, \quad c=m^2+n^2,$$
where $m$ and $n$ are relatively prime positive integers with $m>n$ and $m\not\equiv n\pmod{2}$. In 1956, Je$\acute{s}$manowicz \cite{J55} proposed the following problem.

\begin{con}\label{conj1} The exponential Diophantine equation
\begin{equation}\label{eq1}
(m^2-n^2)^x+(2mn)^y=(m^2+n^2)^z\end{equation}
has only one positive integer solution $(x, \, y,\,z)=(2, \,2,\,2)$. \end{con}

Using elementary methods, Le \cite{L99} showed that if $mn\equiv 2\pmod{4}$ and $m^2+n^2$ is a power of a prime, then Conjecture \ref{conj1} is true. Guo, Le \cite{GL95} applied the theory of linear forms in two logarithms to prove that if $n=3,\, m\equiv2\pmod{4}$ and $m>6000$, then Conjecture \ref{conj1} is true. Takakuwa \cite{Ta96} extended the result of Guo, Le \cite{GL95} by proving that if $n=3,\,7,\,11,\,15$ and $m\equiv2\pmod{4}$, then Conjecture \ref{conj1} is true. Cao \cite{Ca99} also showed that if $m\equiv5\pmod{8}$ and $n\equiv2\pmod{8}$, then Conjecture \ref{conj1} is true.
In 2014, Terai \cite{Te14} showed that if $n=2$,  then Conjecture \ref{conj1} is true without any assumption on $m$.  In 2015, Miyazaki and Terai \cite{MT15} proved some further results.

Recently, Ma and Chen \cite{MC17} proved the following proposition.

\begin{prop} Suppose that $4\not| mn$. Then the equation
$$(m^2-n^2)^x+(2mn)^y=(m^2+n^2)^z, \quad y\ge2,$$
has only the positive integer solution $(x, \, y, \, z)=(2, \, 2, \, 2)$.\end{prop}

Deng and Huang\cite{DHB17}, Deng and Guo  \cite{DG17} proved some theorems for $2||mn$ by using biquadratic character theory and an elementary method. For more results on the conjecture, see \cite{DC98, FM12, L59, HY18, MiBAM09,  M11, M13, M15, M18, MYW14,  TY13, Te99, YH18, ZZ14}.

For the proof of the above Proposition 1.1, Ma and Chen \cite{MC17} used some complicated computations of Jacobi's symbols and a known result of Miyazaki (\cite{M11} Theorem 1.5), which is based on deep results on generalized Fermat equations via sophisticated arguments in the theory of elliptic curves and modular forms. We also note that the proof of the main result in  Terai \cite{Te14} used the same known result of Miyazaki (\cite{M11} Theorem 1.5).

In this paper,  we  present an elementary  proof of Proposition 1.1 by using Jacobi's symbols, however the computations of Jacobi's symbols are more involved here.

\section{Some Lemmas}

For more self-contained, in this section, we provide some simple lemmas which will be used in the proof of Proposition  1.1. The following two results are well-known.
   \begin{lemma}\label{le1}  Let $(u,\,v,\,w)$  be a primitive Pythagorean triple such that $u^2 +v^2 =w^2$, \,  $2|v$  and $w\equiv5\pmod{8}$. Then there exists coprime positive integers $s$ and $t$ with $s>t$, $2||st$ and
   $$u=s^2-t^2, \quad v=2st, \quad w=s^2+t^2.$$
   \end{lemma}
     \begin{proof} The
others being obvious, only $2||st$ needs a proof, this follows from the condition $w\equiv5\pmod{8}$. \end{proof}

   \begin{lemma}\label{le2}The equation
   $x^4-y^4=z^2$ has no nonzero integer solutions.\end{lemma}
For the proof of the above Lemma, we refer to  Mordell \cite{Mo69}.
\begin{prop} \label{th2.1} Let $m, \, n$ be coprime positive integers with $m^2+n^2\equiv5\pmod{8}$ and $m>n$, then the Diophantine equation
  \begin{equation}\label{eqeven} (m^2-n^2)^x+(2mn)^y=(m^2+n^2)^z\end{equation} has only the positive integer solution  $x=y=z=2$ with $2|\gcd(x, \, y)$.\end{prop}

\begin{proof} Let $(x, \, y,\, z)$ be a positive integer solution  of (\ref{eqeven}) with $2|\gcd(x, \, y)$ and $(x, \, y,\, z)\ne(2, \, 2,\, 2)$.  Since $m^2+n^2\equiv5\pmod{8}$ and $(m^2+n^2)^z=(m^2-n^2)^x+(2mn)^y\equiv1\pmod{8}$, we obtain that $2|z$. Put
$$x=2X, \quad y=2Y, \quad z=2Z,$$
then we have
\begin{equation}\label{eqle1}
(m^2-n^2)^X=u^2-v^2, \quad (2mn)^Y=2uv, \quad (m^2+n^2)^Z=u^2+v^2,\end{equation}
where $u, \, v$ are positive integers with $u>v$.
If $Y=1$ and $Z=1$, then it is easy to see that $X=1$, and we are done.

If $Y=1$ and $Z>1$, then we have
$$(2mn)^2=(m^2+n^2)^{2Z}-(m^2-n^2)^{2X}\ge(m^2+n^2)^{Z}+(m^2-n^2)^{X}$$
$$>(m^2+n^2)^{Z}>(2mn)^2,$$
a contradiction.  Finally we consider the case where $Y>1$. If $Y>1$ and $Z$ is even,  we have
$$m^2+n^2\equiv5\pmod{8}, \quad m^2-n^2\equiv\pm5\pmod{8}.$$
Considering equation (\ref{eqeven}) by taking modulo 16, we have
$$(m^2-n^2)^{2X}\equiv1\pmod{16},$$
hence $2|X$, which is impossible by Lemma \ref{le2} since $4|x$, $4|z$ and $(m^2+n^2)^{x}-(m^2-n^2)^{y}=(2mn)^{2Y}$. Therefore $Z$ is odd when $Y>1$. Now
$$(m^2+n^2)^{Z}\equiv5\pmod{8}.$$
It follows from (\ref{eqle1})  that $(m^2+n^2)^Z=s^2+t^2\equiv5\pmod{8}$, hence $2||st$  by Lemma \ref{le1} , which contradicts to $2st=(2mn)^Y$ and $Y>1$. This completes the proof.\end{proof}

\begin{lemma}\label{le3} Let $(x,\, y,\, z)$ be a solution of (\ref{eq1}) with $y\ge2$. Suppose that $2||mn$. Then both $x$ and $z$ are even. \end{lemma}
\begin{proof}  Let $(x,\, y,\, z)$ be a solution of (\ref{eq1}). Since $2||mn$, so $m^2+n^2\equiv5\pmod{8}$ and we have
$$\left(\frac{2mn}{m^2+n^2}\right)=\left(\frac{(m+n)^2}{m^2+n^2}\right)=1, \quad \left(\frac{m^2-n^2}{m^2+n^2}\right)=\left(\frac{2m^2}{m^2+n^2}\right)=\left(\frac{2}{m^2+n^2}\right)=-1.$$
Taking (\ref{eq1}) modulo $m^2+n^2$, we have $\left(\frac{m^2-n^2}{m^2+n^2}\right)^x=\left(\frac{-1}{m^2+n^2}\right)\left(\frac{2mn}{m^2+n^2}\right)^y$, i.e. $(-1)^x=1$, so $2|x$.  In view of $y\ge2$, (\ref{eqle1}) and $4|2mn$,
$$5^z\equiv(m^2+n^2)^z=(m^2-n^2)^x+(2mn)^y\equiv1\pmod{8}.$$
It follows that $z$ is even.\end{proof}

\section{A simple proof of Proposition 1.1}

In this section, we will present an elementary and simple proof of Proposition 1.1.

{\bf A simple proof of Proposition 1.1:} Let $(x, \, y,\, z)$ be a solution of (\ref{eq1}) with $y\ge2$. Noting that $2||mn$, by Lemma \ref{le3}, $2|x$ and $2|z$. If $2|y$, then (\ref{eq1}) has only the solution $(x, \, y,\, z)=(2, \, 2,\, 2)$  by Proposition \ref{th2.1}. Hence we may assume that $2\not|m$ , $2||n, n=2n', 2\not|n'$ and $2\not|y$.
Let $t$ be the positive integer with $2^t||z$. Since $2|x$, we have
$$(4mn')^y=(m^2+4n'^2)^z-(m^2-4n'^2)^x>(m^2+4n'^2)^{z/2}.$$
If $t=1$, then we have $y\ge3$ because $2\not|y$ and $y>1$. If $t>1$, then $y\ge z/2+1\ge t+1$, and thus $2y\ge2(t+1)>t+2$.

Taking modulo $2^{t+3}$ for (\ref{eq1}), we get
$$(m^2-4n'^2)^x+2^{2y}(mn')^y\equiv(m^2-4n'^2)^x\equiv(m^2+4n'^2)^z\equiv1+2^{t+2}\pmod{2^{t+3}},$$
which yields $2^t||x$ since $m^2-4n'^2\equiv\pm5\pmod{8}$.

Let $x=2^tX$ and $z=2^tZ$, where $X$ and $Z$ are positive integers and $2\not|XZ$.

{\bf Case I: $t$ is even}: By (1),  we have
$$(m^2+4n'^2)^{2^tZ}-(m^2-4n'^2)^{2^tX}=$$
$$\left((m^2+4n'^2)^Z-(m^2-4n'^2)^X\right)\prod_{i=0}^{t-1}\left((m^2+4n'^2)^{2^iZ}+(m^2-4n'^2)^{2^iX}\right)=2^{2y}(mn')^y.$$
Since $\gcd(m, n')=1$, it is easy to show that the greatest common divisor of any two terms in the above product is 2 and $(m^2+4n'^2)^Z-(m^2-4n'^2)^X\equiv0\pmod{8}$, hence we have
\begin{equation}\label{eq21}
(m^2+4n'^2)^Z+(m^2-4n'^2)^X=2(m_1n_1)^y\end{equation} and
\begin{equation}\label{eq22}
(m^2+4n'^2)^Z-(m^2-4n'^2)^X=2^{2y-t}(m_2n_2)^y,\end{equation} where $m_i|m$, $n_i|n', i=1, 2$ and $\gcd(m_1, \, m_2)=1, \, \gcd(n_1, \, n_2)=1$.
By (\ref{eq21}) and (\ref{eq22}), we have
\begin{equation}\label{eq23}
(m^2+4n'^2)^Z=(m_1n_1)^y+2^{2y-t-1}(m_2n_2)^y.\end{equation}
In view of (\ref{eq23}), $2y-t-1\ge t+1\ge3$ and $2\not|yZ$, we have
\begin{equation}\label{eq24}
m_1n_1\equiv (m_1n_1)^y\equiv m^2+n^2\equiv5\pmod{8}.\end{equation}
For any prime factor $p$ of $n_1$, by (\ref{eq21}),
$$m^{2Z}+n'^{2X}\equiv0\pmod{p},$$
it follows that $p\equiv1\pmod{8}$. Hence $n_1\equiv1\pmod{8}$, and so $m_1\equiv5\pmod{8}$ by (\ref{eq24}). Similarly, by (\ref{eq22}) we have $m_2\equiv1\pmod{8}$.

On the other hand, since $t$ is even,  it follows from (\ref{eq23}) that
$$\left(\frac{2m_2n_2}{m_1}\right)=\left(\frac{n'^2}{m_1}\right)=1, \quad \left(\frac{2m_2n_2}{n_1}\right)=1$$
and
$$\left(\frac{m_1n_1}{m_2}\right)=\left(\frac{n'^2}{m_2}\right)=1, \quad \left(\frac{m_1n_1}{n_2}\right)=1.$$
In view of $t$ is even, $n_1\equiv m_2\equiv1\pmod{8}$ and $m_1\equiv5\pmod{8}$, we have
\begin{equation}\label{eq25}
\left(\frac{m_2n_2}{m_1}\right)=-1, \quad \left(\frac{m_2n_2}{n_1}\right)=1\end{equation}
and
\begin{equation}\label{eq26}
\left(\frac{m_1n_1}{m_2}\right)=1, \quad \left(\frac{m_1n_1}{n_2}\right)=1.\end{equation}

Let $\left(\frac{m_2}{m_1}\right)=u, \, u\in\{-1, 1\}$, by the first equalities of (\ref{eq25}) and (\ref{eq26}), we have
\begin{equation}\label{eq27}
\left(\frac{n_2}{m_1}\right)=-u, \quad \left(\frac{n_1}{m_2}\right)=u.\end{equation}
Now, by the second equalities of (\ref{eq25}) and  (\ref{eq27}), we get
\begin{equation}\label{eq28}\left(\frac{n_2}{n_1}\right)=u.\end{equation}
By the first equality of (\ref{eq27}) and the second equality of (\ref{eq26}), we have
\begin{equation}\label{eq29}\left(\frac{n_2}{n_1}\right)=-u.\end{equation}
Therefore we derive a contradiction from (\ref{eq28}) and (\ref{eq29}).

{\bf Case II: $t$ is odd.} Similarly, by (1), we have
\begin{equation}\label{eq230}
\left(m^2+4n'^2\right)^{2^{t-i}Z}+\left(m^2-4n'^2\right)^{2^{t-i}X}=2(m_in_i)^y, \,\, i=1, 2, \ldots, t-1,\end{equation}
\begin{equation}\label{eq231}
\left(m^2+4n'^2\right)^{Z}+\left(m^2-4n'^2\right)^{X}=2(m_tn_t)^y\end{equation}
and
\begin{equation}\label{eq232}
\left(m^2+4n'^2\right)^{Z}-\left(m^2-4n'^2\right)^{X}=2^{2y-t}(m_{t+1}n_{t+1})^y.\end{equation}
Similarly, we have $m_1\equiv n_1\equiv \cdots m_{t-1}\equiv n_{t-1}\equiv n_t\equiv m_{t+1}\equiv1\pmod{8}$ and $m_t\equiv m^2+4n'^2\equiv5\pmod{8}$.
By (\ref{eq230}), we have
\begin{equation}\label{eq233}
\left(\frac{2m_in_i}{m+2n'}\right)=1, \,\, i=1, 2, \ldots, t-1.\end{equation}
Since $m_1\equiv n_1\equiv \cdots m_{t-1}\equiv n_{t-1}\equiv1\pmod{8}$, by (\ref{eq233})
\begin{equation}\label{eq234}
\left(\frac{2}{m+2n'}\right)=\left(\frac{n'}{m_i}\right)\left(\frac{m}{n_i}\right), \,\, i=1, 2, \ldots, t-1.\end{equation}

By (\ref{eq231}), we have
\begin{equation}\label{eq235}
\left(\frac{2}{m+2n'}\right)=\left(\frac{2m_tn_t}{m+2n'}\right).\end{equation}
Since $m_t\equiv5\pmod{8}$ and $ n_t\equiv1 \pmod{8}$, by (\ref{eq235})
\begin{equation}\label{eq236}
1=\left(\frac{2n'}{m_t}\right)\left(\frac{m}{n_t}\right)=-\left(\frac{n'}{m_t}\right)\left(\frac{m}{n_t}\right).\end{equation}

By (\ref{eq232}) and $t$ is odd, we have
\begin{equation}\label{eq237}
\left(\frac{2}{m+2n'}\right)=\left(\frac{2m_{t+1}n_{t+1}}{m+2n'}\right).\end{equation}
Since $m_{t+1}\equiv1\pmod{8}$, by (\ref{eq237})
\begin{equation}\label{eq238}
1=\left(\frac{2n'}{m_{t+1}}\right)\left(\frac{m}{n_{t+1}}\right)=\left(\frac{n'}{m_{t+1}}\right)\left(\frac{m}{n_{t+1}}\right).\end{equation}
Combine the three equations (\ref{eq234}), (\ref{eq236}) and (\ref{eq238}), we obtain
$$\left(\frac{2}{m+2n'}\right)^{t-1}=-\prod_{i=1}^{t+1}\left(\frac{n'}{m_i}\right)\left(\frac{m}{n_i}\right)=-\left(\frac{n'}{m}\right)\left(\frac{m}{n'}\right)=-1,$$
contradicts to the fact that $t$ is odd. This completes the proof. $\Box$

 \end{document}